\documentclass[11pt]{amsart}
\input amssym.def
\usepackage{xcolor, amsmath}
\newtheorem{theorem}{Theorem}

\newtheorem{proposition}{Proposition}

\newtheorem{lemma}{Lemma}[section]

\newtheorem*{acknowledgement*}{Acknowledgement}

\begin{document}

\title[Torsion groups of Mordell curves over cubic and sextic fields]{Torsion groups of Mordell curves over\\ cubic and sextic fields }

\author{Pallab Kanti Dey}
\address[]{Pallab Kanti Dey, Ramakrishna Mission Vivekananda Educational and Research Institute, Belur Math, Howrah - 711202, West Bengal, India}
\email[]{pallabkantidey@gmail.com}

\author{Bidisha Roy}
\address[]{Bidisha Roy, Harish-Chandra Research Institute, HBNI, Chhatnag Road, Jhunsi, Prayagraj-211019, Uttar Pradesh, India}
\email[]{bidisharoy@hri.res.in}

\subjclass[2010]{Primary: 11G05, 11R16, 11R21; Secondary: 14H52}

\keywords{Elliptic curves, Torsion group, Number fields}

\begin{abstract}
In this paper, we classify torsion groups of rational Mordell curves explicitly over cubic fields as well as over sextic fields. Also, we classify torsion groups of Mordell curves over cubic fields and for Mordell curves over sextic fields, we produce all possible torsion groups.
\end{abstract}

\maketitle
\section{Introduction}

Let $K$ be a number field and $E$ be an elliptic curve defined over $K$. Then,   by the Mordell-Weil theorem,  the group $E(K)$ of $K$-rational points is a finitely generated abelian group. Hence, by the structure theorem of finitely generated abelian groups, we  have $E(K) \cong T \oplus \mathbb{Z}^r$, where $T$ is the torsion subgroup of $E(K)$ and $r \geq 0$ is the rank of $E$ over $K$. Sometimes we may write $T = {E(K)}_{tors}$. 
 
\smallskip

Let $\Phi (d)$ be the set of all possible isomorphic torsion structures $E(K)_{tors}$, where $K$ runs through all number fields $K$ of degree $d$ and $E$ runs through all elliptic curves over $K$. By the uniform
boundedness theorem \cite{mer}, $\Phi(d)$ is a finite set. Then the problem is to determine $\Phi(d)$.

\smallskip



 When $K = \mathbb {Q}$, in \cite{maz}, Mazur proved that 
 $$ \Phi(1) = \{\mathbb{Z}/m\mathbb{Z}: \quad 1 \leq m \leq 12, \quad m \neq 11\} \cup \{\mathbb{Z}/2\mathbb{Z} \oplus \mathbb{Z}/{2m} \mathbb{Z}: 1 \leq m \leq 4\}.$$ 

\smallskip

By a result of  Kamienny \cite{kam} and by another result of Kenku and Momose \cite{kemo}, it is known that
$$ 
\Phi(2) = \{\mathbb{Z}/m\mathbb{Z}: \quad 1 \leq m \leq 18, \quad m \neq 17\} \cup \{\mathbb{Z}/2\mathbb{Z} \oplus \mathbb{Z}/{2m} \mathbb{Z}: \quad 1 \leq m \leq 6\}$$ 
$$\cup \{\mathbb{Z}/3\mathbb{Z} \oplus \mathbb{Z}/{3m}\mathbb{Z}: \quad  m = 1,2\} \cup \{\mathbb{Z}/4\mathbb{Z}\}.
$$



\smallskip

Moreover in \cite{jkp},  it has been proved that if we let the quadratic fields vary, then all of the $26$ torsion subgroups described above appear infinitely often. However it is still unknown about the set $\Phi(d)$ for $d \geq 3$ in general.

\smallskip

In analogy to $\Phi(d)$, let $\Phi_{\mathbb{Q}}(d)$ be the subset of $\Phi(d)$ such that $H \in \Phi_{\mathbb{Q}}(d)$ if there is an elliptic curve $E$ over $\mathbb{Q}$ and a number field $K$ of degree $d$ such that $E(K)_{tors} \simeq  H$. Note that when $K=\mathbb{Q}$, we see that $\Phi(1) = \Phi_{\mathbb{Q}}(1)$.




\smallskip

One of the first general result is due to Najman \cite{naj2}, who determined $\Phi_{\mathbb{Q}}(d)$ completely for $d = 2$ and $3$. More precisely, he proved that

$$
\Phi_{\mathbb{Q}}(2)= \{\mathbb{Z}/m\mathbb{Z}: \quad 1 \leq m \leq 16, \quad m \neq 11,13,14\} \cup \{\mathbb{Z}/2\mathbb{Z}\oplus \mathbb{Z}/2m\mathbb{Z}: \quad 1 \leq m \leq 6\}$$
$$\cup \{\mathbb{Z}/3\mathbb{Z}\oplus \mathbb{Z}/3m\mathbb{Z}: \quad m = 1, 2\}\cup \{\mathbb{Z}/4\mathbb{Z}\oplus \mathbb{Z}/4\mathbb{Z}\}
$$

and

$$
\Phi_{\mathbb{Q}}(3) = \{\mathbb{Z}/m\mathbb{Z}: 1 \leq m \leq 21, m \neq 11,15,16,17,19,20\}$$
$$\cup \{\mathbb{Z}/2\mathbb{Z} \oplus \mathbb{Z}/{2m} \mathbb{Z}: 1 \leq m \leq 7, m \neq 5,6\}.
$$

 Very recently, Gonz\'{a}lez-Jim\'{e}nez and Najman  \cite{gn} completely determined the set $\Phi_{\mathbb{Q}}(4)$. Before that, Gonz\'{a}lez-Jim\'{e}nez \cite{egj} completely described the set $\Phi_{\mathbb{Q}}(5)$. 
Recently, H.B. Daniels and Gonz\'{a}lez-Jim\'{e}nez \cite{jd} have given a partial answer to the classification of $\Phi_{\mathbb{Q}}(6)$. In particular, they studied what groups (up to isomorphism) can occur as the torsion subgroup
of $E/\mathbb{Q}$ base-extended to a sextic field. In \cite{gn}, it was also proved that $\Phi_{\mathbb{Q}}(7) = \Phi(1)$ and $\Phi_{\mathbb{Q}}(d) = \Phi(1)$ for any integer $d$ not divisible by $2,3,5$ and $7$.

\smallskip

In the case of elliptic curves with complex multiplication, we denote by ${\Phi}^{CM}(d)$ and ${\Phi}^{CM}_{\mathbb{Q}}(d)$ the analogue of the sets $\Phi(d)$ and ${\Phi}_{\mathbb{Q}}(d)$ respectively but restricting to elliptic curves with complex multiplication. The set ${\Phi}^{CM}(1)$ was completely determined by Olson \cite{ols}. Also ${\Phi}^{CM}(2)$ and ${\Phi}^{CM}(3)$ was determined by M{\" u}ller et al.  \cite{muller} and by Zimmer et al.  \cite{fswz, pwz} respectively. Recently, Clark et al. \cite{ccrs} have computed the sets ${\Phi}^{CM}(d)$ for $2 \leq d \leq 13$.


 
 \smallskip
 
 Family of Mordell curves over number field $K$ are of the form $y^2 = x^3 + c$  where $c \in K$. In particular, it is known that this family of curves are CM elliptic curves. In the case of Mordell curves, we denote by ${\Phi}^{M}(d)$ and ${\Phi}^{M}_{\mathbb{Q}}(d)$ the analogue of the sets $\Phi(d)$ and ${\Phi}_{\mathbb{Q}}(d)$ respectively but restrict to Mordell curves. The set ${\Phi}^{M}(1)$ was determined completely in \cite{kna}. Recently, in \cite{dey2}, the set ${\Phi}^{M}_{\mathbb{Q}}(d)$ was computed for $d=2$ and for all $d\geq 5$ with $gcd(d,6) = 1$. 

\smallskip

Motivated by the above results, in this paper, we determine the set 
${\Phi}^{M}(d)$ and ${\Phi}^{M}_{\mathbb{Q}}(d)$ completely for $d=3$ and $6$. Moreover, we completely describe the conditions whenever a member from  the sets ${\Phi}^{M}_{\mathbb{Q}}(3), {\Phi}^{M}_{\mathbb{Q}}(6)$ and ${\Phi}^{M}(3)$, is actually appearing as a torsion subgroup.

\section{Main results}
Consider a family of Mordell curves of the form $y^2 = x^3 + c$  where $c\in \mathbb Q$. By a rational transformation, it is enough to assume that $c$ is an integer. For this family of curves, we derive precise torsion subgroup of $E(K)$ for  any cubic or sextic field $K$.

For an elliptic curve $E: y^2 = x^3 +c$ with $c\in \mathbb Z$,  we write $c = c_1 t^6$ for some sixth power-free integer $c_1$ and for some nonzero integer $t$. Then $(x,y)$ is a point on the elliptic curve $E_1: y^2 = x^3 + c_1$ if and only if $(t^2x, t^3y)$ is a point on $E$. Thus, it is enough to assume that $c$ is a sixth power-free integer to compute the torsion subgroup of $E(K)$ for some number field $K$. Here we prove the following results.

\smallskip

\begin{theorem}\label{cubicQ}
Let $E: y^2 = x^3 + c$  be a Mordell curve for any $6$-th power-free integer $c$. Then  $${\Phi}^{M}_{\mathbb{Q}}(3) = \{\mathbb{Z}/9\mathbb{Z}, \mathbb{Z}/6\mathbb{Z},\mathbb{Z}/3\mathbb{Z},\mathbb{Z}/2\mathbb{Z},\mathcal{O}\}.$$ Moreover, if $K/\mathbb{Q}$ is a cubic field and $T$ is the torsion subgroup of $E(K)$ then $T$ is isomorphic to one of the following groups:

\begin{enumerate}
\item[(1)] $\mathbb{Z}/9\mathbb{Z}$, \quad if $c = 16$ and  $K = \mathbb{Q}(r)$ with $r$  satisfying  $r^3-3r^2+1=0$,

\item[(2)] $\mathbb{Z}/6\mathbb{Z},
 \left\{
\begin{array}{l} 
\mbox{if }c \mbox{ is both square and cube in }K,
\\   \mbox{or }c = -27 \mbox{ and }4 \mbox{ is a cube in } K,
 \end{array}
 \right.
 $
\item[(3)] $\mathbb{Z}/3\mathbb{Z},
 \left\{
\begin{array}{l} 
\mbox{if }c (\neq 16) \mbox{ is a square but not a cube in }K,\\ \mbox{or } c = 16 \mbox{ and } K \neq \mathbb{Q}(r) \mbox{ with }r \mbox{ satisfying }  r^3-3r^2+1=0,\\ \mbox{or }4c  \mbox{ is a cube and }-3c \mbox{ is a square in }K,
\end{array}
\right.$
\item[(4)] 
$\mathbb{Z}/2\mathbb{Z}, 
\left\{
\begin{array}{l}
\mbox{if } c(\neq -27) \mbox{ is a cube but not a square in }K,\\
\mbox{or }c = -27\mbox{ but }4 \mbox{ is not a cube in K},
\end{array}
\right.
$  
\item[(5)] $\{\mathcal{O}\}$, otherwise.
\end{enumerate}
\end{theorem}

\smallskip

\begin{theorem}\label{cubicK}
Let $E: y^2 = x^3 + c$  be any Mordell curve defined over a cubic field. Then  $${\Phi}^{M}(3) = \{\mathbb{Z}/9\mathbb{Z},\mathbb{Z}/6\mathbb{Z},\mathbb{Z}/3\mathbb{Z},\mathbb{Z}/2\mathbb{Z},\mathcal{O}\}.$$ 
Moreover,  if $K/\mathbb{Q}$ is a cubic field and $T$ is the torsion subgroup of $E(K)$ then $T$ is isomorphic to one of the following groups:
\begin{enumerate}
\item[(1)] $\mathbb{Z}/9\mathbb{Z},
 \left\{
\begin{array}{l} 
\mbox{if }c \mbox{ is a square and }4c \mbox{ is a cube in }K \mbox{ with }K = \mathbb{Q}(r) \mbox{ where } r\\ \quad \mbox{ satisfying }  r^3-3r^2+1=0,
 \end{array}
 \right.
$ 
 \item[(2)] $\mathbb{Z}/6\mathbb{Z},
 \left\{
\begin{array}{l} 
\mbox{if }c \mbox{ is both square and cube in }K,
\\   \mbox{or }c,4 \mbox{ are cubes }  \mbox{ and }-3c \mbox{ is a square in }K,
 \end{array}
 \right.
 $
\item[(3)] $\mathbb{Z}/3\mathbb{Z},
 \left\{
\begin{array}{l} 
\mbox{if }c \mbox{ is a square and }4c \mbox{ is a cube in }K \mbox{ with }K \neq \mathbb{Q}(r) \mbox{ where } r\\ \quad \mbox{ satisfying }  r^3-3r^2+1=0,\\
\mbox{or }c \mbox{ is a square but }4c \mbox{ is not a cube in }K,\\ 
\mbox{or }4c  \mbox{ is a cube and }-3c \mbox{ is a square in }K,
\end{array}
\right.$

\item[(4)] 

$\mathbb{Z}/2\mathbb{Z},
\left\{
\begin{array}{l} 
\mbox{if } c \mbox{ is a cube but not a square in }K \mbox{ and }-3c \mbox{ is not a square in }K,\\
\mbox{or } c \mbox{ is a cube but not a square in }K \mbox{ and }4 \mbox{ is not a cube in }K,
\end{array}
\right.$  
\item[(5)] $\{\mathcal{O}\}$, otherwise.

\end{enumerate}
\end{theorem}

\begin{theorem}\label{sexticQ}
Let $E: y^2 = x^3 + c$  be a Mordell curve for any $6$-th power-free integer $c$. Then  $${\Phi}^{M}_{\mathbb{Q}}(6) = {\Phi}^{M}_{\mathbb{Q}}(3) \cup \{\mathbb{Z}/9\mathbb{Z} \oplus \mathbb{Z}/3\mathbb{Z}, \mathbb{Z}/6\mathbb{Z} \oplus \mathbb{Z}/6\mathbb{Z}, \mathbb{Z}/6\mathbb{Z} \oplus \mathbb{Z}/2\mathbb{Z}, \mathbb{Z}/3\mathbb{Z} \oplus \mathbb{Z}/3\mathbb{Z}, \mathbb{Z}/2\mathbb{Z} \oplus \mathbb{Z}/2\mathbb{Z}\}.$$
Moreover, if $K/\mathbb{Q}$ be any sextic field and $T$ is the torsion subgroup of $E(K)$, then $T$ is isomorphic to one of the following groups:
\begin{enumerate}

\item[(1)] $\mathbb{Z}/9\mathbb{Z},
 \left\{
\begin{array}{l} 
\mbox{if }c \mbox{ is a square and }4c \mbox{ is a cube in }K, \omega \notin K \mbox{ and }\mathbb{Q}(r) \subset K \mbox{ with } r\\ \quad \mbox{ satisfying }  r^3-3r^2+1=0,
 \end{array}
 \right.
$
\item[(2)] $\mathbb{Z}/6\mathbb{Z},
 \left\{
\begin{array}{l} 
\mbox{if }c \mbox{ is both square and cube in }K \mbox{ and }\omega \notin K,\\
 \mbox{or } c,4 \mbox{ are cubes in } K \mbox{ and } -3c \mbox{ is a square in }K  \mbox{ with }\omega \notin K,
 \end{array}
 \right.
 $
\item[(3)] $\mathbb{Z}/3\mathbb{Z},
 \left\{
\begin{array}{l}
\mbox{if }c \mbox{ is a square but not a cube in }K \mbox{ and }4c \mbox{ is not a cube in }K,\\
\mbox{or }c \mbox{ is a square and }4c \mbox{ is a cube in }K, \omega \notin K \mbox{ and }\mathbb{Q}(r) \not\subset K \mbox{ with } r\\ \quad \mbox{ satisfying }  r^3-3r^2+1=0,\\
\mbox{or }  c \mbox{ is not a cube in }K, 4c \mbox{ is a cube and } -3c \mbox{ is a square in }K \mbox{ with }\omega \notin K,\\
\end{array}
\right.$
\item[(4)] 
$\mathbb{Z}/2\mathbb{Z}, 
\left\{
\begin{array}{l}
 \mbox{if } c \mbox{ is a cube but not a square in }K \mbox{ and }4 \mbox{ is not a cube in }K,\\
 \mbox{or } c \mbox{ is a cube but not a square in }K \mbox{ and }-3c \mbox{ is not a square in }K,\\
 \end{array}
\right.$

\item[(5)] $\mathbb{Z}/9\mathbb{Z} \oplus \mathbb{Z}/3\mathbb{Z},
 \left\{
\begin{array}{l} 
\mbox{if }c \mbox{ is a square and }4c \mbox{ is a cube in }K, \omega \in K \mbox{ and }\mathbb{Q}(r) \subset K\mbox{ with } r\\ \quad \mbox{ satisfying }  r^3-3r^2+1=0,
 \end{array}
 \right.
$ 

\item[(6)] 
$\mathbb{Z}/6\mathbb{Z} \oplus \mathbb{Z}/6\mathbb{Z}$, if $c$ is both square and cube in $K$  with $4$ is a cube in $K$ and $\omega \in K$,
  
\item[(7)]
$\mathbb{Z}/6\mathbb{Z} \oplus \mathbb{Z}/2\mathbb{Z}$,
 if $c$ is both square and cube in $K$ with $4$ is not a cube in $K$ and $ \omega \in K$,

\item[(8)] 
$\mathbb{Z}/3\mathbb{Z} \oplus \mathbb{Z}/3\mathbb{Z},
 \left\{
\begin{array}{l} 
\mbox{if }c \mbox{ is a square and not a cube in }K, 4c \mbox{ is a cube in }K, \omega \in K \mbox{ and }\\ \quad \mathbb{Q}(r) \not\subset K\mbox{ with } r \mbox{ satisfying }  r^3-3r^2+1=0,
 \end{array}
 \right.
$ 
\item[(9)] $\mathbb{Z}/2\mathbb{Z} \oplus \mathbb{Z}/2\mathbb{Z}$, if $c$ is a cube but not a square in $K$ with $\omega \in K$,
\item[(10)] $\{\mathcal{O}\}$, otherwise.
\end{enumerate}
\end{theorem}


\begin{theorem}\label{sexticK}
If $K/\mathbb{Q}$ be any sextic field and $E: y^2 = x^3 + c$  be a Mordell curve for any $c \in K$. Then, $${\Phi}^{M}(6) = {\Phi}^{M}_{\mathbb{Q}}(6) \cup \{\mathbb{Z}/19\mathbb{Z},   \mathbb{Z}/7\mathbb{Z}, \mathbb{Z}/14\mathbb{Z} \oplus \mathbb{Z}/2\mathbb{Z}\}.$$
\end{theorem}

\section{Preliminaries}

\smallskip
For any elliptic curve $E$ defined over a field $K$ and for any positive integer $n$, we define
$$
E(K)[n] = \{P = (x,y) \in E(K): nP = \mathcal{O}\} \cup \{\mathcal{O}\}.
$$

\bigskip

Following lemma gives an ideaa about the structure of torsion subgroup of an elliptic curve over quadratic field. 

\begin{lemma} [\cite{gt}, Corollary 4]
\label{twist}
Let $E$ be an elliptic curve defined over a number field $K$. Let $E^d$ be the $d$-quadratic twist of $E$ for some $d \in K^*/(K^*)^2$. Then for any odd positive integer $n$,
$$
E(K(\sqrt{d}))[n] \cong E(K)[n] \times E^{d}(K)[n].
$$
\end{lemma}

\bigskip

Now, we state three important lemmas which give information regarding the cardinality of torsion subgroups of elliptic curves defined over finite fields. 

\bigskip

\begin{lemma} [\cite{dey2}, Corollary 1]
\label{dey1} 
Let $E: y^2 = x^3 + c$ be an elliptic curve for some non-zero integer $c$.  Let  $p\equiv 2\pmod{3}$ be a prime for which $E$ has good reduction mod $p$. Then, we have 
$$
|\overline{E}(\mathbb F_{p^n})| =
\left\{\begin{array}{cl} 
 p^n + 1, & \mbox{ if } n  \mbox{ is odd, } \\
 (p^{\frac{n}{2}} + 1)^2, & \mbox{ if } n\equiv 2\pmod{4}.
\end{array}\right.$$
\end{lemma}

\bigskip

\begin{lemma}\label{dey2}
Let $E_p:y^2=x^3+c$ be an elliptic curve defined over $\mathbb{F}_{p^3}$ and assume that $p \equiv 2 \pmod 3$ be an odd prime. Then we have 
$$
|{E_p}(\mathbb F_{p^3})| = 
 p^3 + 1.
$$
\end{lemma}

\begin{proof}
The multiplicative group $(\mathbb{F}_{p^3})^{\times}$ has order $p^3 - 1$. Since $p \equiv 2 \pmod 3$, we observe that there is no element of order $3$ in $(\mathbb{F}_{p^3})^{\times}$. Therefore, the homomorphism $a \rightarrow a^3$ is a bijection on $(\mathbb{F}_{p^3})^{\times}$. In particular, for each $y \in \mathbb{F}_{p^3}$, the element $y^2 - c$ has a unique cubic root, which we can consider as $x$. Thus, in this process we obtain $p^3$ number of points on ${E_p}(\mathbb F_{p^3})$. With the  additional point at infinity, we see that ${E_p}(\mathbb F_{p^3})$ has $p^3 + 1$ points.
\end{proof}

\bigskip

\begin{proposition} [\cite{was}, Theorem 4.3, p. 98]\label{criterion}
Let $q = p^n$ be a power of a prime $p$ and let $N = q + 1 -a$. Then, there is an elliptic curve $E$ defined over $\mathbb{F}_q$ such that $|E(\mathbb{F}_q)| = N$ iff $|a| \leq 2 \sqrt{q}$ and $a$ satisfies one of the following.
\begin{enumerate}
\item[(1)] $\gcd(a, p) = 1$,
\item[(2)]$n$ is even and $a = \pm2 \sqrt{q}$,
\item[(3)] $n$ is even, $p \not\equiv 1 \pmod 3$ and $a = \pm \sqrt{q}$,
\item[(4)] $n$ is odd, $p = 2$ or $3$ and $a = \pm p^{(n+1)/2}$,
\item[(5)] $n$ is even, $p \not\equiv 1 \pmod 4$ and $a = 0$,
\item[(6)] $n$ is odd and $a = 0$.
\end{enumerate}
\end{proposition}

\bigskip

Following two results provide criteria to determinine the supersingularity of elliptic curves. 

\begin{proposition} [\cite{sil}, Theorem 4.1, p. 148]\label{super1}
Let $\mathbb{F}_q$ be a finite field of characteristic $p \geq 3$. Let $E$ be an elliptic curve over $\mathbb{F}_q$ given by a Weierstrass equation
$$
E: y^2 = f(x),
$$
where $f(x) \in \mathbb{F}_q[x]$ is a cubic polynomial with distinct roots in $\overline{\mathbb{F}}_q$. Then $E$ is supersingular if and only if the coefficient of $x^{p-1}$ in ${f(x)}^{(p-1)/2}$ is zero.
\end{proposition}

\bigskip

\begin{lemma} [\cite{was}, Proposition 4.31, p. 130]\label{super2}
Let $E$ be an elliptic curve over $\mathbb{F}_q$, where $q$ is a power of the prime $p$. Let $a = q + 1 - |E(\mathbb{F}_q)|$. Then $E$ is supersingular if and only if $a \equiv 0 \pmod p$.
\end{lemma}

\bigskip

In the following proposition, we record the nature of the reduction map for elliptic curves over a given number field.

\begin{proposition} [\cite{dey1}, Proposition 4]   \label{reduction}
Let $E$ be an elliptic curve defined over $K$ and  $T$ be the torsion subgroup of $E(K)$. Let $\mathcal{O}_K$ be the ring of integers in $K$ and let $\mathcal{P}$ be a prime ideal lying in $\mathcal{O}_K$. If $E$ has good reduction at $\mathcal{P}$, let $\phi$ be the reduction modulo $\mathcal{P}$ map on $T$ which means the reduction map
$\phi: T \longrightarrow \bar{E}(\mathcal{O}_{K}/\mathcal{P})$ is defined as $P = (x,y) \mapsto \bar {P} = (\bar{x},\bar{y})$ if $P \neq \mathcal{O}$ and $\mathcal{O} \mapsto \bar{\mathcal O}$. 
Then, the reduction map $\phi$ is an injective homomorphism except finely many prime ideal $\mathcal{P}$.
\end{proposition}

\bigskip




Now we talk about the Kubert-Tate normal form of an elliptic curve, which we will use for proving Theorem \ref{sexticK}.
\begin{theorem}[\cite{kub}]
Let $E$ be an elliptic curve over a field $K$ and $P \in E(K)$ be a point of order at least $4$. Then $E$ can be written of the form
\begin{equation}\label{1}
y^2 + (1-c)xy - by = x^3 - bx^2
\end{equation}
for some $b,c \in K$ with $P=(0,0)$.
\end{theorem}

\smallskip

The curve defined in \eqref{1} is called as the {\it Kubert-Tate normal form} of $E$, and we denote this curve simply by $E(b,c)$. The $j$-invariant of this elliptic curve is
\begin{equation}\label{kub}
j(b,c) = \frac{(16b^2+8b(1-c)(c+2)+(1-c)^4)^3}{b^3(16b^2-b(8c^2+20c-1)-c(1-c)^3)}.
\end{equation}

\bigskip

\section{proof of theorem \ref{cubicQ}}
Throughout this section, $K/\mathbb{Q}$ stands for a cubic number field.
We denote a rational Mordell curve of the form $y^2 = x^3 + c$ for some $6$-th power-free integer $c$, simply by $E$. We also denote $T$ as the torsion subgroup of $E(K)$.

\begin{lemma}\label{Qq}
For any prime $ q \geq 5$, there is no element in $T$ of order $q$.
\end{lemma}
\begin{proof}
Suppose there exists an element of order $q$ in $T$. Then $q$ divides $|T|$. Since $\gcd (2, 3q)=1$, by Dirichlet's theorem on primes in arithmetic progressions, there exist infinitely many primes $p \equiv 2 \pmod {3q}$. Therefore there is a prime $p\equiv 2 \pmod{ 3q}$ such that $E$ has good reduction at $p$. We consider such a prime and let $p \mathcal{O}_K = {\mathcal{P}}_{1}^{e_1} {\mathcal{P}}_{2}^{e_2}{\mathcal{P}}_{3}^{e_3}$ be the ideal decomposition in $\mathcal{O}_K$ where $\mathcal{P}_1,\mathcal{P}_2,\mathcal{P}_3$ are prime ideals in $\mathcal{O}_K$ lying above $p$ and $0 \leq e_i \leq 1$. For $ i=1,2, 3$, if $f_i$'s are the residual degree of $\mathcal{P}_i$,  then we know that $e_1f_1+e_2f_2 +e_3f_3 = 3$. Hence we have a prime ideal $\mathcal{P}_j$ with $f_j =1 $ or $3$ for some $ j=1,2 , 3$. 

\smallskip

Now, we consider the reduction mod $\mathcal{P}_j$ map. Since $p\equiv 2 \pmod 3$, by Lemma \ref{dey1}, we get $|\overline{E}(\mathcal{O}_K/\mathcal{P}_j)| = p^{f_j} + 1$. Hence by Proposition \ref{reduction},  $q$ divides $|T|$ and thus   $q \mid (p^{f_j} + 1)$. Since $p\equiv 2 \pmod 3$, we have $0 \equiv p^{f_j}+1 \equiv 2^{f_j}+1 \pmod q$. Since $f_j =1 $ or $3$, we see that $3 \equiv 0 \pmod q$ or $9 \equiv 0 \pmod q$, which is a contradiction to $q\geq 5$.
\end{proof}

\smallskip

\begin{lemma}
\label{Q2}
Let $E$ be a rational Mordell curve. Then,  
$$
E(K)[2] \cong 
\begin{cases}
\mathbb{Z} / 2 \mathbb{Z}, & \text{ if } c \mbox{ is a cube in } K,\\
\mathcal{O}, & \text{ otherwise}.
\end{cases}
$$
\end{lemma}
\begin{proof}
If $ P=(x,y)$ is a point of order $2$, then $y=0$ and $x$ satisfies the polynomial equation $x^3+c=0$. Hence $c$ has to be a cube in $K$. Thus, in this case, $E(K)[2] \cong \mathbb{Z} / 2 \mathbb{Z}$; otherwise it is trivial.
\end{proof}

\smallskip

\begin{lemma}
\label{Q3}
Let $E$ be a rational Mordell curve. Then,  
$$
E(K)[3] \cong 
\begin{cases}
\mathbb{Z} / 3 \mathbb{Z}, & \text{ if } c \mbox{ is a square in } K,\\\mathbb{Z} / 3 \mathbb{Z}, & \mbox{ if } -3c \mbox{ is a square in }K \mbox{ and } 4c \mbox{ is a cube in } K,\\
\mathcal{O}, & \text{ otherwise}.
\end{cases}
$$
\end{lemma}
\begin{proof}
If $ P=(x,y)$ is a point of order $3$, then $x$ satisfies the polynomial equation  $x(x^3+4c)=0$. 

\smallskip

If $x=0$, then $c$ is a square in $K$. In that case, we conclude that $E(K)[3]\cong \mathbb{Z} / 3 \mathbb{Z}$. 

\smallskip

If $x \neq 0$, then we get $x^3+4c=0$ and hence $y^2 = -3c$. Thus $4c$ is a cube in $K$ and $-3c$ is a square in $K$. Hence, in this case, we have $E(K)[3] \cong \mathbb{Z} / 3 \mathbb{Z}$.
\end{proof}

\smallskip

\begin{lemma}
\label{Q6}
Let $E$ be a rational Mordell curve. Then,  
$$
E(K)[6] \cong 
\begin{cases}
\mathbb{Z} / 6 \mathbb{Z}, & \text{ if } c \mbox{ is a square as well as cube in } K,\\\mathbb{Z} / 6 \mathbb{Z}, & \mbox{ if } c =-27 \mbox{ and }4 \mbox{ is a cube in } K,
\\ \mathbb{Z} / 3 \mathbb{Z}, & \mbox{ if } c \mbox{ is a square but not a cube in }K,
\\ \mathbb{Z} / 3 \mathbb{Z}, & \mbox{ if } -3c \mbox{ is a square in }K \mbox{ and }4c \mbox{ is a cube in }K,
\\\mathbb{Z} / 2 \mathbb{Z}, & \mbox{ if } c \neq -27 \mbox{ is a cube but not a square in }K,\\ \mathbb{Z} / 2 \mathbb{Z}, & \mbox{ if } c = -27 \mbox{ but }4 \mbox{ is not a cube in }K,\\ \mathcal{O}, & \text{ otherwise}.
\end{cases}
$$
\end{lemma}

\begin{proof}

\smallskip

\noindent{\bf Case 1.} ($c$ is a square as well as cube in $K$)

\smallskip

Since $c$ is a square in $K$, by Lemma \ref{Q3}, there are two points of order $3$ in $E(K)[6]$. Also, since $c$ is a cube in $K$, by Lemma \ref{Q2}, there is only one point of order $2$ in $E(K)[6]$. Hence,  it has an element of order $6$. Thus, we conclude that $E(K)[6] \cong \mathbb Z/6 \mathbb Z$.

\smallskip

\noindent{\bf Case 2.} ($c$ is a cube in $K$ but not a square in $K$)

\smallskip

Since $c$ is a cube in $K$, $E(K)[6]$ has  only one element of order $2$ by Lemma \ref{Q2}. If $E(K)[6]$ has an element of order $3$, by Lemma \ref{Q3}, we see that  $-3c$ is a square and $4c$ is a cube in $K$. Since both $c$ and $4c$ are cubes, we see that $4$ is a cube in $K$. If $c$ is not a cube of an integer, we observe that $K = \mathbb{Q}(c^{1/3})$, which is not possible as $c$ is not a power of $2$. Hence $c$ is a cube of an integer. Also $-3c$ is a square of an integer. Hence $c=-27$. In this case, $E(K)[6] \cong \mathbb Z/6 \mathbb Z$ when $4$ is a cube in $K$. Thus, if $c \neq -27$ or $4$ is not a cube in $K$, we conclude that $E(K)[6] \cong \mathbb Z/2 \mathbb Z$.

\smallskip

\noindent{\bf Case 3.} ($c$ is a square but not a cube in $K$)

\smallskip

Since $c$ is not a cube in $K$, by Lemma \ref{Q2}, $E(K)[6]$ has no element of order $2$. Since $c$ is a square, by Lemma \ref{Q3}, we conclude that $E(K)[6] \cong \mathbb Z/3 \mathbb Z$.

\smallskip

\noindent{\bf Case 4.} ($c$ is neither a square nor a cube in $K$)

Since $c$ is not a cube in $K$, by Lemma \ref{Q2}, $E(K)[6]$ has no element of order $2$. If $E(K)[6]$ has an element of order $3$, by Lemma \ref{Q3}, we see that  $-3c$ is a square in $K$ and $4c$ is a cube in $K$. Thus, we conclude that  $E(K)[6] \cong \mathbb Z/3 \mathbb Z$.
\end{proof}

\smallskip

\begin{lemma}\label{Q4}
$T$ has no element of order $4$.
\end{lemma}
\begin{proof}
Suppose $P = (x,y) \in T$ is an element of order $4$. Then, we get  $$y(2P) = 0 \Longleftrightarrow x^6 + 20cx^3 -8c^2 = 0 \Longleftrightarrow x^3 = -10c \pm 6c\sqrt{3}.$$ Since $x^3 \in K$, we see that $-10c \pm 6c\sqrt{3} \in K$ which in turn implies $\sqrt{3} \in K$, which is a contradiction as $K$ is a cubic field. This proves the lemma.
\end{proof}

\smallskip

\begin{lemma}\label{Q9}
$T$ has an element of order $9$ if and only if $c = 16$ and $K = \mathbb{Q}(r)$ with $r$ satisfying the relation $r^3 - 3r^2 + 1 = 0$. 
\end{lemma}

\begin{proof}
Let $P = (x,y)$ be an element of order $9$ in $E(K)_{tors}$. Then $3P$ is a point of order $3$ in $E(K)_{tors}$ and hence  we get $x(3P) ({x(3P)}^3+4c)=0$.

If $({x(3P)}^3+4c)=0$, then $c = - \frac{a^3}{4}$ for some $a \in K$, where $x(3P) = a$. Now, let $\alpha$ be the slope of the line joining $P$ and $2P$. Then by the addition formula, we have
$$ {\alpha}^2 -x -\frac{x(x^3-8c)}{4(x^3+c)}=a= x(3P),$$ 
where $\alpha = \left( \frac{7x^6-4cx^3+16c^2}{6xy(x^3+4c)}\right)^2$.

This equation can be written explicitly as 
$$
x^9 - 9ax^8 + 24a^3x^6 + 18a^4x^5 + 3a^6x^3 - 9a^7x^2 - a^9 = 0.
$$
By substituting $x = at$ for some $t \in K$, the above equation becomes
$$
t^9 - 9t^8 + 24t^6 + 18t^5 + 3t^3 - 9t^2 - 1 = 0.
$$
Using {\it magma}, we see that the polynomial $f(X) = X^9 - 9X^8 + 24X^6 + 18X^5 + 3X^3 - 9X^2 - 1$ is irreducible over $\mathbb{Q}$. Since $[K: \mathbb{Q}]=3$, the relation $t^9 - 9t^8 + 24t^6 + 18t^5 + 3t^3 - 9t^2 - 1=0$ is impossible. Therefore we get $x(3P) ^3+4c \neq 0$.

Thus, we conclude $x(3P)=0$, which implies that $c$ is a square in $K$. Again, by the addition formula, we get $$ \left( \frac{7x^6-4cx^3+16c^2}{6xy(x^3+4c)} \right)^2 -x -\frac{x(x^3-8c)}{4(x^3+c)}=0.$$ 
The above equation reduces to $$ (x^3 + c)(x^9 -96cx^6+48c^2 x^3+64c^3)=0.
$$ 

 If $x^3+c=0$ then we have $2P=\mathcal{O}$, which is a contradiction to $P$ is of order $9$. Hence we get 
 \begin{equation}\label{change}
x^9 -96cx^6+48c^2 x^3+64c^3=0.
\end{equation}
 Putting $4c = t \in K$, the above equation further reduces to
 $$x^9-24tx^6+3t^2x^3+t^3=0.$$
 This equation can be rewritten as 
 \begin{equation}\label{change1}
 (x^3+t)^3=27tx^6,
 \end{equation}
 which shows that $t$ is a cube in $K$, say $v^3$ for some $v \in K$. Since cube root of unity, $\omega \notin K$, from equation \eqref{change1}, we have
 $$x^3+v^3 = 3vx^2.$$
 Substituting $\frac{x}{v} = r \in K$, the above equation reduces to
 \begin{equation}\label{change2}
 r^3-3r^2+1=0.
 \end{equation}
Since the polynomial $r^3-3r^2+1$ is an irreducible polynomial and $K$ is a cubic field, we see that $K = \mathbb{Q}(r)$. $\mathbb{Q}(r)/\mathbb{Q}$ is a normal extension as the equation \eqref{change2} has three real roots.

Also we have
$\frac{y^2}c = \frac{x^3}c+1$. Since $c$ is a square in $K$ with $c = \frac{v^3}{4}$ and $r = \frac{x}{v}$, by putting $\gamma = y/\sqrt{c} \in K$, we get
\begin{equation}\label{change3} 
\gamma^2 = 4r^3+1. 
\end{equation}
Combining equations \eqref{change2} and \eqref{change3}, we have $$ \gamma^6 -99 \gamma^4 +243 \gamma^2 -81=0,
$$
which further implies $$  (\gamma^3-9\gamma^2-9\gamma+9) (\gamma^3+9\gamma^2-9\gamma-9)=0.$$  
By letting $f(X)=X^3-9X^2-9X+9$, we see that either $f(\gamma)=0$ or $f(-\gamma)=0$. Without loss of generality, we assume $f(\gamma)=0$. Since $f(X)$ is irreducible over $\mathbb{Q}$, we conclude that $K=\mathbb{Q}(r) = \mathbb{Q}(\gamma)$. 

Hence, if $T$ has a point of order $9$ in $K$, then $c$ is a square in $K$ and $4c$ is a cube in $K$ where $K = \mathbb{Q}(r)$ with $r$ satisfying the relation $r^3 - 3r^2 + 1 = 0$. 

Conversely, if $c$ is a square in $K$ and $4c$ is a cube in $K$ where $K = \mathbb{Q}(r)$ with $r$ satisfying the relation $r^3 - 3r^2 + 1 = 0$, then we can show that $((4c)^{1/3}r, \pm c^{1/2} \gamma_1),((4c)^{1/3}r, \pm c^{1/2} \gamma_2),$ $((4c)^{1/3}r, \pm c^{1/2} \gamma_3)$ are all points of order $9$ in $K$, where  $\gamma_1, \gamma_2, \gamma_3$ are roots of the equation ${\gamma}^3 - 9{\gamma}^2 - 9{\gamma} + 9 = 0$.

If $4c$ is not a cube of an integer then we have $\mathbb{Q}((4c)^{1/3}) = \mathbb{Q}(r)$, which is not possible as $\mathbb{Q}(r)/\mathbb{Q}$ is a normal extension and $\mathbb{Q}((4c)^{1/3})/\mathbb{Q}$ is not a normal extension. Thus $4c$ is a cube of an integer. Since $c$ is a square of an intger also, we conclude that $c = 16$. 
\end{proof}

\begin{lemma}\label{Q18}
$T$ has no element of order $18$ and $27$.
\end{lemma}
\begin{proof}
Suppose $T$ has an element of order $18$ or $27$. Hence there exist an element of order $9$ in $T$. Thus, by Lemma \ref{Q9}, we have $c=16$ and $K= \mathbb{Q}(r)$ with $r$ satisfying the relation $r^3-3r^2+1=0$. In this case, by using {\it magma}, we see that $T \cong \mathbb Z/9 \mathbb Z$, which is a contradiction. Hence $T$ has no element of order $18$ or $27$.
\end{proof}

\bigskip

\noindent{\it Proof of Theorem \ref{cubicQ}}. Combining Lemmas \ref{Qq}, \ref{Q6}, \ref{Q4}, \ref{Q9} and \ref{Q18},  we have the desired result. \qed

\bigskip

\section{Proof of theorem \ref{cubicK}}
Throughout this section, $K/\mathbb{Q}$ stands for a cubic number field.
We denote a Mordell curve of the form $y^2 = x^3 + c$ with $c \in K$, simply by $E$. Also, we denote $T$ as the torsion subgroup of $E(K)$.

\begin{lemma}\label{Kq}
For any odd prime $q \geq 5$, there is no element in $T$ of order $q$.
\end{lemma}

\begin{proof}
Suppose there exists an element of order $q$ in $T$. Then $q$ divides $|T|$. Since $\gcd (2, 3q)=1$, by Dirichlet's theorem on primes in arithmetic progressions, there exist infinitely many primes $ p \equiv 2 \pmod {3q}$. Therefore there is a prime $p\equiv 2 \pmod{ 3q}$ such that $E$ has good reduction at some prime ideal lying above $p$. We  assume that $p \mathcal{O}_K = {\mathcal{P}}_{1}^{e_1} {\mathcal{P}}_{2}^{e_2}{\mathcal{P}}_{3}^{e_3}$ is the ideal decomposition in $\mathcal{O}_K$ where $\mathcal{P}_1,\mathcal{P}_2,\mathcal{P}_3$ are prime ideals in $\mathcal{O}_K$ lying above $p$ and $0 \leq e_i \leq 1$. If $f_i$'s are then residual degree of $\mathcal{P}_i$ for $ i=1,2, 3$, then we know that $e_1f_1+e_2f_2 +e_3f_3 = 3$. Therefore there exists a prime ideal $\mathcal{P}_j$ such that $f_j =1 $ or $3$ for some $ j=1,2,3$.

\smallskip
Now, we consider the reduction mod $\mathcal{P}_j$ map.  Note that $|\mathcal{O}_K/\mathcal{P}_j| = p^{f_j}$, where $f_i = 1$ or $3$. Since $p\equiv 2 \pmod 3$, by Lemmas \ref{dey1} and \ref{dey2}, we get $|\overline{E}(\mathcal{O}_K/\mathcal{P}_j)| = p^{f_j} + 1$. As $q$ divides $|T|$, by Proposition \ref{reduction}, we conclude that $q \mid (p^{f_j} + 1)$. Since $p\equiv 2 \pmod 3$, we get $0 \equiv p^{f_j}+1 \equiv 2^{f_j}+1 \pmod q$. Since $f_j =1 $ or $3$, we have either $3 \equiv 0 \pmod q$ or $9 \equiv 0 \pmod q$ respectively, which is a contradiction to $q \geq 5$.
\end{proof}

\smallskip

\begin{lemma}
\label{K2}
Let $E$ be a Mordell curve over $K$. Then,  
$$
E(K)[2] \cong 
\begin{cases}
\mathbb{Z} / 2 \mathbb{Z}, & \text{ if } c \mbox{ is a cube in } K,\\
\mathcal{O}, & \text{ otherwise}.
\end{cases}
$$
\end{lemma}
\begin{proof}
Proof is similar to the proof of Lemma \ref{Q2} and we omit the proof here.
\end{proof}

\smallskip

\begin{lemma}
\label{K3}
Let $E$ be a Mordell curve over $K$. Then,  
$$
E(K)[3] \cong 
\begin{cases}
\mathbb{Z} / 3 \mathbb{Z}, & \text{ if } c \mbox{ is a square in } K,\\\mathbb{Z} / 3 \mathbb{Z}, & \mbox{ if } -3c \mbox{ is a square in }K \mbox{ and } 4c \mbox{ is a cube in } K,\\
\mathcal{O}, & \text{ otherwise}.
\end{cases}
$$
\end{lemma}
\begin{proof}
Proof is similar to the proof of Lemma \ref{Q3} and we omit the proof here.
\end{proof}

\begin{lemma}
\label{K6}
Let $E$ be a Mordell curve over $K$. Then,  
$$
E(K)[6] \cong 
\begin{cases}
\mathbb{Z} / 6 \mathbb{Z}, & \text{ if } c \mbox{ is a square as well as cube in } K,\\\mathbb{Z} / 6 \mathbb{Z}, & \mbox{ if } -3c  \mbox{ is a square in }K \mbox{ and }c,4c \mbox{ are cubes in } K,
\\ \mathbb{Z} / 3 \mathbb{Z}, & \mbox{ if } c \mbox{ is a square but not a cube in }K,
\\ \mathbb{Z} / 3 \mathbb{Z}, & \mbox{ if } -3c \mbox{ is a square in }K \mbox{ and }4c \mbox{ is a cube in }K \mbox{ but }c \mbox{ is not a cube in }K,
\\\mathbb{Z} / 2 \mathbb{Z}, & \mbox{ if } c  \mbox{ is a cube but not a square in }K \mbox{ and }-3c \mbox{ is not a square in }K,\\ \mathbb{Z} / 2 \mathbb{Z}, & \mbox{ if } c \mbox{ is a cube but not a square in }K \mbox{ and }4c \mbox{ is not a cube in }K,\\ \mathcal{O}, & \text{ otherwise}.
\end{cases}
$$
\end{lemma}

\begin{proof}

\noindent{\bf Case 1.} ($c$ is a square as well as cube in $K$)

\smallskip

Since $c$ is a square in $K$, by Lemma \ref{K3}, there are two points of order $3$ in $E(K)[6]$. Also, since $c$ is a cube in $K$, by Lemma \ref{K2}, there is only one point of order $2$ in $E(K)[6]$. Hence,  it has an element of order $6$. Thus, we conclude that $E(K)[6] \cong \mathbb Z/6 \mathbb Z$.

\smallskip

\noindent{\bf Case 2.} ($c$ is a cube in $K$ but not a square in $K$)

\smallskip

Since $c$ is a cube in $K$, $E(K)[6]$ has  only one element of order $2$ by Lemma \ref{K2}. If $E(K)[6]$ has an element of order $3$, by Lemma \ref{K3}, we see that  $-3c$ is a square in $K$ and $4c$ is a cube in $K$. In this case, $E(K)[6] \cong \mathbb Z/6 \mathbb Z$. Thus, if $-3c$ is not a square in $K$ or $4c$ is not a cube in $K$, we conclude that $E(K)[6] \cong \mathbb Z/2 \mathbb Z$.

\smallskip

\noindent{\bf Case 3.} ($c$ is a square but not a cube in $K$)

\smallskip

Since $c$ is not a cube in $K$, by Lemma \ref{K2}, $E(K)[6]$ has no element of order $2$. Since $c$ is a square, by Lemma \ref{K3}, we conclude that $E(K)[6] \cong \mathbb Z/3 \mathbb Z$.

\smallskip

\noindent{\bf Case 4.} ($c$ is neither a square nor a cube in $K$)

Since $c$ is not a cube in $K$, by Lemma \ref{K2}, $E(K)[6]$ has no element of order $2$. If $E(K)[6]$ has an element of order $3$, by Lemma \ref{K3}, we see that  $-3c$ is a square in $K$ and $4c$ is a cube in $K$. Thus, we conclude that  $E(K)[6] \cong \mathbb Z/3 \mathbb Z$.
\end{proof}

\begin{lemma}\label{K4}
$T$ has no element of order $4$.
\end{lemma}
\begin{proof}
Proof is similar to the proof of Lemma \ref{Q4} and we omit the proof here.
\end{proof}

\begin{lemma}\label{K9}
$T$ has an element of order $9$ if and only if $c$ is a square in $K$, $4c$ is a cube in $K$ and $K = \mathbb{Q}(r)$ with $r$ satisfying the relation $r^3 - 3r^2 + 1 = 0$.
\end{lemma}
\begin{proof} 
Proof is similar to the proof of Lemma \ref{Q9}.
\end{proof}

\smallskip

\begin{lemma}\label{K27}
$T$ has no element of order $27$.
\end{lemma}
\begin{proof}

We assume that there exists an element of order $27$ in $T$. Then $27$ divides $|T|$. Since $\gcd (2, 27)=1$, by Dirichlet's theorem on primes in arithmetic progression, there exist infinitely many primes $ p \equiv 2 \pmod {27}$. Therefore there is a prime $p\equiv 2 \pmod{ 27}$ such that $E$ has good reduction at a prime ideal lying above $p$. Let $p \mathcal{O}_K = {\mathcal{P}}_{1}^{e_1} {\mathcal{P}}_{2}^{e_2}{\mathcal{P}}_{3}^{e_3}$ be the ideal decomposition in $\mathcal{O}_K$ where $\mathcal{P}_1,\mathcal{P}_2,\mathcal{P}_3$ are prime ideals in $\mathcal{O}_K$ lying above $p$ and $0 \leq e_i \leq 1$. If $f_i$'s are then residual degree of $\mathcal{P}_i$ for $ i=1,2, 3$, then we know that $e_1f_1+e_2f_2 +e_3f_3 = 3$. Then there exists a prime ideal $\mathcal{P}_j$ such that $f_j =1 $ or $3$ for some $ j=1,2,3$.

\smallskip
 Since $f_j \in \{ 1,3\}$, we know $|\overline{E}(\mathcal{O}_K/\mathcal{P}_j)| \in \{ p+1, p^3+1\}$ by the Lemmas \ref{dey1} and \ref{dey2}. 
Again since the reduction map $\phi: T \longrightarrow \bar{E}(\mathcal{O}_{K}/\mathcal{P}_j)$ is injective except finitely many primes, we conclude  that $ |T|$ divides $|\bar{E}(\mathcal{O}_K/\mathcal{P}_j)|$ and hence $27$ divides $|\bar{E}(\mathcal{O}_K/\mathcal{P}_j)|$. This is impossible as $p \equiv 2 \pmod{27}$. Hence there is no element of order $27$ in $T$.
\end{proof}

\smallskip

\begin{lemma}\label{K18}
$T$ has no element of order $18$.
\end{lemma}
\begin{proof}
Assume that $T$ has a point of order $18$. Then it has a point of order $2$, which forces $c = a^3$ for some $a \in K$. Since $T$ has a point of order $9$, say $P = (x,y)$, then we have
$$x^{12} -95a^3x^9-48a^6x^6+112a^9x^3+64a^{12}=0.,$$
from the proof of Lemma \ref{Q9}.

\smallskip

After substituting $t=\frac{x}{a} \in K$, the above equation reduces to $$t^{12} -95t^9-48t^6 +112t^3 +64=0\\ \Rightarrow (t^3+1)(t^9-96t^6+48t^3+64)=0.$$ 

\smallskip
Since $(t^3+1) \neq 0$, we have $(t^9-96t^6+48t^3+64)=0$. Now consider the polynomial $f(t) = (t^9-96t^6+48t^3+64)$ and using {\it magma},
we get that $f(t)$ is an irreducible polynomial in $\mathbb{Z}[t]$. Hence $[\mathbb{Q}(t) : \mathbb{Q}] = 9$, which is a contradiction as $t \in K$ and $[K : \mathbb Q] = 3$.
Hence there is no point of order $18$ in $T$.
\end{proof}

 \smallskip
 
 \noindent{\it Proof of Theorem \ref{cubicK}}. Combining Lemmas \ref{Kq}, \ref{K6}, \ref{K4}, \ref{K9}, \ref{K27} and \ref{K18},  we have the desired result. \qed
 
\bigskip

 \section{Proof of Theorem \ref{sexticQ}}
 
 Throughout this section $K$ stands for a sextic field. We denote a rational Mordell curve of the form $y^2 = x^3 + c$ for some integer $c$, simply by $E$. We also denote $T$ as the torsion subgroup of $E(K)$.
 \begin{lemma}\label{rq} 
Let $q > 3$ be any prime. Then there does not exist any element of order $q$  in $T$.
\end{lemma}
\begin{proof}
Suppose there exists an element of order $q$ in $T$. Then $q$ divides $|T|$. Sine $\gcd (2, 3q)=1$, by Dirichlet's theorem on primes in arithmetic progressions, there exist infinitely many primes $ p \equiv 2 \pmod {3q}$. Therefore there is a prime $p\equiv 2 \pmod{ 3q}$ such that $E$ has good reduction at $p$. We consider such a prime and assume that $p \mathcal{O}_K = {\mathcal{P}}_{1}^{e_1} \cdots {\mathcal{P}}_{6}^{e_6}$ is the ideal decomposition in $\mathcal{O}_K$ where $\mathcal{P}_1,\ldots,\mathcal{P}_6$ are prime ideals in $\mathcal{O}_K$ lying above $p$ and $0 \leq e_i \leq 1$. If $f_i$'s are then residual degree of $\mathcal{P}_i$ for $ i=1,\ldots,6$, then we know that $\sum_{i = 1}^{6} e_{i}f_{i} = 6$. Therefore there exists a prime ideal $\mathcal{P}_j$ such that $f_j$ is either $1 $  or $2$ or $3$ or $6$ for some $ j=1,2,\ldots, 6$. 

\smallskip

Now, we consider the reduction mod $\mathcal{P}_j$ map. Now  we have $|\mathcal{O}_K/\mathcal{P}_j| = p^{f_j}$, where $f_j = 1,2,3$ or $6$. Since $p \equiv 2 \pmod{3}$, we get $|\overline{E}(\mathcal{O}_K/\mathcal{P}_j)| \mid (p^3 + 1)^2$ by Lemma \ref{dey1}. Then,  we  conclude that $q \mid (p^3 + 1)$ by Proposition \ref{reduction}. But we also have $p \equiv 2 \pmod{q}$, which implies $p^3 + 1 \equiv 9 \pmod{q}$. It is a contradiction as $q > 3$. Hence there does not exist any point of order $q > 3$.
\end{proof}

   \begin{lemma}\label{r2}
 Let $E$ be a rational Mordell curve. Then
$$E(K)[2] \cong \left\{\begin{array}{ll}
\mathbb{Z}/2\mathbb{Z}\oplus \mathbb{Z}/2\mathbb{Z}, & \text{ if }c \mbox{ is a cube in }K \mbox{ and } \omega \in K\\
        \mathbb{Z}/2\mathbb{Z}, & \text{ if }c \mbox{ is a cube in }K \mbox{ and } \omega \notin K\\
        
         \mathcal{O}, & \text{ otherwise,}
        \end{array}\right.$$
        where $\omega$ is a cube root of unity.
\end{lemma}
\begin{proof}
If $ P=(x,y)$ is a point of order $2$, then $y=0$ and $x$ satisfies the polynomial equation $x^3+c=0$. Hence $c$ has to be a cube in $K$.

If $\omega \in K$, then $E(K)[2] \cong \mathbb{Z} / 2 \mathbb{Z} \times \mathbb{Z} / 2 \mathbb{Z}$. If $\omega \notin K$, then $E(K)[2] \cong \mathbb{Z} / 2 \mathbb{Z}$.
\end{proof}

\smallskip

\begin{lemma}\label{r3}
Let $E$ be a rational Mordell curve. Then
$$E(K)[3] \cong  \left\{\begin{array}{ll}
\mathbb{Z}/3\mathbb{Z}\oplus\mathbb{Z}/3\mathbb{Z}, & \text{ if }4c \mbox{ is a cube in }K \mbox { and } c,\omega \text{ are squares in } K,\\
\mathbb{Z}/3\mathbb{Z}, & \text{ if }4c \mbox{ is a cube in }K\mbox{ and }-3c \text{ is a square but }c \mbox{ is not square in }K,\\
\mathbb{Z}/3\mathbb{Z}, & \text{ if }c \mbox{ is a square in }K\mbox{ and }4c \text{ is not cube or }\omega \mbox{ is not square in }K,\\
\mathcal{O} & \text{ otherwise.} \
\end{array}\right.$$
\end{lemma}

\begin{proof}
If $P= (x,y)$ be a point of order $3$ in $T$, then $$x(x^3+4c)=0.
$$

If $x=0$, then $y^2=c$ and hence $c$ is a square in $K$. Therefore, $(0, \pm \sqrt{c})$ are the only possible points of order $3$ in $T$.

\smallskip

If $x \neq 0$, then we have $x^3+4c=0$ and hence $y^2 = -3c$. Therefore $(-(4c)^{\frac{1}{3}}, \pm \sqrt{-3c})$, $(-(4c)^{\frac{1}{3}}\omega, \pm \sqrt{-3c})$ or $(-(4c)^{\frac{1}{3}}\omega^2, \pm \sqrt{-3c})$ are the possible points of order $3$ in $T$.

Hence, if $4c$ is a cube and $c,\omega$ are  squares in $K$, then $E(K)[3] \cong \mathbb{Z} / 3 \mathbb{Z} \times \mathbb{Z} / 3 \mathbb{Z}$. If $4c$ is a cube and $-3c$ is a  square in $K$ but $c$ is not a square in $K$, then $E(K)[3] \cong \mathbb{Z} / 3 \mathbb{Z}$. Also if $c$ is a square in $K$ and $4c$ is not a cube or $\omega$ is not a square in $K$, then $E(K)[3] \cong \mathbb{Z} / 3 \mathbb{Z}$. Combining all cases we have the desired result.
\end{proof}

\smallskip

\begin{lemma}\label{r4}
$T$ does not have any element of order $4$.
\end{lemma}

\begin{proof}
Let $T$ has an element of order $4$. Hence $T$ has an element of order $2$, which forces $c$ to be a cube, say $a^3$ for some $a \in K$. We assume that $P = (x,y)$ is an element of order $4$ in $T$. Then we note  $y(2P) = 0 \Longleftrightarrow x^6 + 20a^3x^3 -8a^6 = 0 \Longleftrightarrow x^3 = -10c \pm 6a^3\sqrt{3}$. Hence $x^3=(-10 \pm 6\sqrt{3})a^3$. Since $[K:\mathbb{Q}]=6$, we have $x = (-1\pm \sqrt{3})a \in K$. Thus, we conclude that $K =K_1(\sqrt{3})$, where $K_1 = \mathbb{Q}(a)$ is a cubic subfield of $K$. 

\smallskip

Since  $y \in K$, we can write $y = t_1 + t_2 \sqrt{3}$  for some $t_1, t_2 \in K_1$. Since $y^2 = x^3 + c \in K=K_1(\sqrt{3})$, we get $(t_1 +t_2 \sqrt{3})^2 = (-10c \pm 6c \sqrt{3}) +c \in K$. Since $\{1, \sqrt{3}\}$ is a basis of $K$ over $K_1$, we get two relations which are $t_1^2 + 3 t_2^2 = -9c$ and $t_1 t_2 = \pm 3c$. These two relations together imply $t_1^2 + 3t_2^2 \pm 3t_1 t_2 = 0$. Putting  $t = \displaystyle\frac{t_1}{t_2}\in K_1$, we get $$t^2 \pm 3t + 3 = 0 \Longrightarrow t = \frac{\mp 3 \mp \sqrt{-3}}{2}.$$ This implies that $\sqrt{-3} \in K_1$, which is a contradiction as $K_1$ is a cubic extension over $\mathbb{Q}$. Hence, we conclude that there does not exist any element of order $4$ in $E(K)_{tors}$.
\end{proof}

\smallskip

\begin{lemma}\label{r9} $T$ has an element of order $9$ if and only if $c$ is a square, $4c$ is a cube in $K$ and $\mathbb{Q}(r) \subset K$ with $r$ satisfying the relation $r^3-3r^2+1=0$..
\end{lemma}

\begin{proof} 
Let $P = (x,y)$ be an element of order $9$ in $E(K)_{tors}$. 

Then following the proof of Lemma \ref{Q9}, $x$ satisfies the polynomial equation 
 $$
 (x^3+t)^3=27tx^6,
 $$
 where $t = 4c \in K$ with $c$ is a square in $K$.
 From this equation we observe that $t$ is a cube in $K$, say $v^3$ for some $v \in K$. 
 Then we can write
 $$(x^3+v^3 -3vx^2)(x^3+v^3 -3v\omega x^2)(x^3+v^3 -3v\omega^2 x^2)=0,$$
 where $\omega$ is a cube root of unity.
 Substituting $\frac{x}{v} = r_1 \in K, \frac{x}{v\omega} = r_2 \in K$ and $\frac{x}{v\omega^2} = r_3 \in K$, the above equation reduces to
 \begin{equation}\label{change21}
 r^3-3r^2+1=0,
 \end{equation}
 where $r \in K$ is one of $r_1,r_2$ and $r_3$. Since the polynomial $r^3-3r^2+1$ is an irreducible polynomial and $K$ is a cubic field, we see that $\mathbb{Q}(r) \subset K$ is a cubic extension over $\mathbb{Q}$. $\mathbb{Q}(r)/\mathbb{Q}$ is a normal extension as the equation \eqref{change21} has three real roots.

Also we have
$\frac{y^2}c = \frac{x^3}c+1$. Since $c$ is a square in $K$ with $c = \frac{v^3}{4}$ and $r = \frac{x}{v}$, by putting $\gamma = y/\sqrt{c} \in K$, we get
\begin{equation}\label{change22} 
\gamma^2 = 4r^3+1. 
\end{equation}
Combining equations \eqref{change21} and \eqref{change22}, we have $$ \gamma^6 -99 \gamma^4 +243 \gamma^2 -81=0,
$$
which further implies, $$  (\gamma^3-9\gamma^2-9\gamma+9) (\gamma^3+9\gamma^2-9\gamma-9)=0.$$  
By letting $f(X)=X^3-9X^2-9X+9$, we see that either $f(\gamma)=0$ or $f(-\gamma)=0$. Without loss of generality, we assume $f(\gamma)=0$. Since $f(X)$ is irreducible over $\mathbb{Q}$, we conclude that $\mathbb{Q}(\gamma) \subset K$ is a cubic extension over $\mathbb{Q}$. Since $r \in \mathbb{Q}(\gamma)$, we observe that $\mathbb{Q}(\gamma)=\mathbb{Q}(r)$. 

Hence, if $T$ has a point of order $9$ in $K$, then $c$ is a square and $4c$ is a cube in $K$ where $\mathbb{Q}(r) \subset K$ with $r$ satisfying the relation $r^3 - 3r^2 + 1 = 0$. 

Conversely, if $c$ is a square in $K$ and $4c$ is a cube in $K$ where $\mathbb{Q}(r) \subset K$ with $r$ satisfying the relation $r^3 - 3r^2 + 1 = 0$, then we can show that $((4c)^{1/3}r, \pm c^{1/2} \gamma_1),((4c)^{1/3}r, \pm c^{1/2} \gamma_2),$ $((4c)^{1/3}r, \pm c^{1/2} \gamma_3)$ are points of order $9$ in $K$, where  $\gamma_1, \gamma_2, \gamma_3$ are roots of the equation ${\gamma}^3 - 9{\gamma}^2 - 9{\gamma} + 9 = 0$. In fact if $\omega \notin K$, then there are $6$ points of order $9$ in $T$ and if $\omega \in K$, then there are $18$ points of order $9$ in $T$.
\end{proof}

\smallskip

\begin{lemma}\label{r27}
$T$ has no element of order $27$.
\end{lemma}
\begin{proof}
Suppose $T$ has an element of order $27$. Hence there exist an element of order $9$ in $T$. Thus, by Lemma \ref{r9}, we see that $c$ is a square, $4c$ is a cube in $K$ and $K= \mathbb{Q}(r)$ with $r$ satisfying the relation $r^3-3r^2+1=0$. Hence $K = \mathbb{Q}(r)(\sqrt{d})$ for some $d \in (\mathbb{Q}(r)/{\mathbb{Q}(r)}^2)^*$. Then, by Lemma \ref{twist}, we have
$$
E(K)[27] \cong E(\mathbb{Q}(r))[27] \times E^{d}(\mathbb{Q}(r))[27].
$$ 
Since, by Lemma \ref{K27}, there are no points of order $27$ in $E(\mathbb{Q}(r))$ and $E^d(\mathbb{Q}(r))$, we conclude that there are no elements of order $27$ in $E(K)$.
\end{proof}

\smallskip

\begin{lemma}\label{r18}
$T$ has no element of order $18$.
\end{lemma}
\begin{proof} 
Proof is similar to the proof of Lemma \ref{K18}.
\end{proof}

\smallskip

\noindent{\it Proof of Theorem \ref{sexticQ}}. By Lemmas \ref{rq}, \ref{r2}, \ref{r3}, \ref{r4}, \ref{r9}, \ref{r27} and \ref{r18}, we conclude that the only possible orders for the nontrivial torsion points in $T$ are $2,3,6$ and $9$.

\smallskip

\noindent{\bf Case 1.} ($c$ is a cube and a square in $K$)

\smallskip

 \noindent {\bf Subcase $a$}. ($ \omega \notin K$)

Since $c$ is a square in $K$, there are two points of order $3$ by Lemma \ref{r3}. Again $c$ is a cube in $K$ implies that there are only one point of order $2$ in $T$ by Lemma \ref{r2}. Hence, $T \cong \mathbb Z/6\mathbb Z$.

\smallskip

\noindent {\bf Subcase $b$}. ($\omega \in K$)

Since $c$ is a square in $K$, there are eight points of order $3$ by Lemma \ref{r3}, if $4^{\frac{1}{3}} \in K$. Also $c$ is a cube in $K$ provides that there are three points of order $2$ in $T$ by Lemma \ref{r2}. Hence, in this case, $T\cong \mathbb Z/6\mathbb Z \oplus \mathbb Z/6\mathbb Z$. If $4^{\frac{1}{3}} \notin K$, then there are only two points of order $3$ by Lemma \ref{r3} and we get $T \cong \mathbb Z/6\mathbb Z \oplus \mathbb Z/2\mathbb Z$.

\smallskip

\noindent{\bf Case 2.} ($c$ is a cube, but not a square in $K$)

\smallskip

In this case, write $c = a^3$ for $a \in K$.

\smallskip 
\noindent {\bf Subcase $a$}. ($ \omega \in K$)

Since $c$ is a cube in $K$, there are three points of order $2$ in $T$  by Lemma \ref{r2}. Also we observe that there does not exist any element of order $3$ by Lemma \ref{r3}. Hence, $T \cong \mathbb Z/2\mathbb Z \oplus \mathbb Z/2\mathbb Z$. 

\smallskip

\noindent {\bf Subcase $b$}.  ($\omega \notin K$)

In this case, $(-a,0)$ is the only point of order $2$ in $T$ by Lemma \ref{r2}. If $-3c$ is a square in $K$ and $4^{\frac{1}{3}} \in K$, then there are two points of order $3$ by Lemma \ref{r3}. Since $T$ is abelian, it has an element of order $6$ and hence $T\cong \mathbb Z/6\mathbb Z$. If $-3c$ is not a square in $K$ or $4^{\frac{1}{3}} \notin K$, then there does not exist any element of order $3$ in $T$ by Lemma \ref{r3} and we get $T \cong \mathbb Z/2\mathbb Z$.

\smallskip

\noindent{\bf Case 3.} ($c$ is a square, but not a cube in $K$)

\smallskip

At first we observe that there are no elements of order $2$ in $T$.

\smallskip

\noindent {\bf Subcase $a$}. ($ \omega \in K$)

Since $c$ is square in $K$, there are $8$ points of order $3$ in $T$  by Lemma \ref{r3} whenever $4c$ is a cube in $K$. If $\mathbb{Q}(r) \subset K$ with $r$ satisfying $r^3-3r^2+1=0$, then by Lemma \ref{r9}, there are $18$ points of order $9$ in $T$ and thus $T \cong \mathbb Z/9\mathbb Z \oplus \mathbb Z/3\mathbb Z$. If $\mathbb{Q}(r) \not\subset K$, then there are no points of order $9$ in $T$ and Hence $T \cong \mathbb Z/3\mathbb Z \oplus \mathbb Z/3\mathbb Z$. If $4c$ is not a cube in $K$, then there are two points of order $3$ in $T$  by Lemma \ref{r3} and we get $T \cong \mathbb Z/3\mathbb Z$. 

\smallskip

\noindent {\bf Subcase $b$}. ($\omega \notin K$)

In this case, there are two elements of order $3$ in $T$ by Lemma \ref{r3}. If $4c$ is a cube in $K$ and $\mathbb{Q}(r) \subset K$ with $r$ satisfying $r^3-3r^2+1=0$, then by Lemma \ref{r9}, there are $6$ points of order $9$ in $T$ and thus $T \cong \mathbb Z/9\mathbb Z$. If $4c$ is a cube in $K$ and $\mathbb{Q}(r) \not\subset K$, then there are no points of order $9$ in $T$ and Hence $T \cong \mathbb Z/3\mathbb Z$. If $4c$ is not a cube in $K$, then we have $T \cong \mathbb Z/3\mathbb Z$.

\smallskip

\noindent{\bf Case 4.} ($c$ is neither a square nor a cube in $K$)

\smallskip

 Since $c$ is not a cube, there are no elements of order $2$ in $T$. If $4c$ is a cube in $K$ and $-3c$ is a square in $K$, then there are two elements of order $3$ in $T$ and in that case,  $T \cong \mathbb Z/3\mathbb Z$.

\smallskip

Hence combining all the cases, Theorem \ref{sexticQ} follows. \qed


\section{Proof of Theorem \ref{sexticK}}

 If $K$ varies over all sextic number fields and $E$ varies over all elliptic curves with complex multiplication over $K$, then Clark {\it et al.}\cite{ccrs} computed the following all possible collection of torsion subgroups;
 \begin{equation}\label{6cm}
 E(K)_{tors} \in
 \left\{
\begin{array}{l} 
\mathbb{Z}/m\mathbb{Z} \mbox{ for }m=1,2,3,4,6,7,9,10,14,18,19,26,
\\ \mathbb{Z}/2\mathbb{Z} \oplus \mathbb{Z}/m\mathbb{Z} \mbox{ for }m=2,4,6,14,\\
\mathbb{Z}/3\mathbb{Z} \oplus \mathbb{Z}/m\mathbb{Z} \mbox{ for }m=3,6,9,\\
\mathbb{Z}/6\mathbb{Z} \oplus \mathbb{Z}/6\mathbb{Z}.
 \end{array}
 \right.
 \end{equation}

From now onwards, throughout this section, $K$ stands for a sextic numer field and we denote the Mordell curve of the form $y^2 = x^3 + c$ for $c \in K$, simply by $E$.

\smallskip

Since any Mordell curve belongs to the family of elliptic curves with complex multiplication, \eqref{6cm} provides all possible collection for $E(K)_{tors}$.

\smallskip

If we restrict $c \in \mathbb{Q}$, then by Theorem \ref{sexticQ} we have already proved
$$E(K)_{tors} \in \Phi^{M}_{\mathbb{Q}}(6)=
 \left\{
\begin{array}{l} 
\mathbb{Z}/m\mathbb{Z} \mbox{ for }m=1,2,3,6,9,
\\ \mathbb{Z}/2\mathbb{Z} \oplus \mathbb{Z}/m\mathbb{Z} \mbox{ for }m=2,6,\\
\mathbb{Z}/3\mathbb{Z} \oplus \mathbb{Z}/m\mathbb{Z} \mbox{ for }m=3,9,\\
\mathbb{Z}/6\mathbb{Z} \oplus \mathbb{Z}/6\mathbb{Z}.
 \end{array}
 \right.
 $$

\smallskip

Now, our task is to eliminate some groups listed in \eqref{6cm}  and we show that the other groups in the above listing \eqref{6cm} can occur as $E(K)_{tor}$.

\begin{lemma}\label{s4}
The groups $\mathbb{Z}/4\mathbb{Z}$ and $\mathbb{Z}/2\mathbb{Z} \oplus \mathbb{Z}/4\mathbb{Z}$ do not appear as $E(K)_{tors}$, for any $K$.
\end{lemma}

\begin{proof}
It is enough to prove that there is no  element of order $4$ in $E(K)_{tors}$. By the similar approach of the proof of Lemma \ref{r4}, one can prove the result and we omit the proof here. 
\end{proof}

\smallskip

\begin{lemma}\label{s5}
The groups $\mathbb{Z}/10\mathbb{Z}$ and $\mathbb{Z}/26\mathbb{Z}$ do not appear  as $E(K)_{tors}$, for any $K$.
\end{lemma}
\begin{proof}
It is enough to prove that there do not exist any element of order $5$ and $13$ in $E(K)_{tors}$. Let us assume that there exists a point of  order $\ell$ in $E(K)_{tors}$, where $\ell$ is either $5$ or $13$. Then, $\ell$ divides $|E(K)_{tors}|$. Now, we set $q=13$. Since   $\gcd(5,12q) = 1$, by Dirichlet's theorem on primes in arithmetic progressions, there are infinitely many primes  $p \equiv 5 \pmod{12q}$ . We consider such a prime and also assume $p \mathcal{O}_K = {\mathcal{P}}_{1}^{e_1} \cdots {\mathcal{P}}_{6}^{e_6}$ is the ideal factorization in $\mathcal{O}_K$ where $\mathcal{P}_1,\ldots,\mathcal{P}_6$ are prime ideals in $\mathcal{O}_K$ lying above $p$ and with $0 \leq e_i \leq 1$. Also we know that $\sum_{i = 1}^{6} e_{i}f_{i} = 6$ where $f_i$'s are residual degree for $\mathcal{P}_i$'s. Hence we can choose a prime ideal for which residual degree is either $1,2,3$ or $6$.

\smallskip

Let $\mathcal{P}_i$ be such a prime ideal and consider the reduction modulo $\mathcal{P}_i$ map. Note that $|\mathcal{O}_K/\mathcal{P}_i| = p^{f_i}$, where $f_i = 1,2,3$ or $6$. 

\smallskip

Since $2$ and $3$ both divide $6$, it is enough to calculate $|\overline{E}(\mathcal{O}_K/\mathcal{P}_i)|$ only for $f_i = 6$. Since $p \equiv 2 \pmod 3$, by Proposition \ref{super1}, we see that $E$ is a supersingular elliptic curve over $\mathbb{F}_{p^6}$. Therefor, we get $E(\mathbb{F}_{p^6}) = p^6+1-a$ with $|a| \leq 2p^3$. Also, by Lemma \ref{super2}, we get $p \mid a$ and by Proposition \ref{criterion}, we have either $a=\pm p^3$ or $\pm 2p^3$. 

\smallskip

 Also, for $f_i = 1,2,3$ or $6$, we have either $|\overline{E}(\mathcal{O}_K/\mathcal{P}_i)|$ divides $(p^3 \pm 1)^2$ or $|\overline{E}(\mathcal{O}_K/\mathcal{P}_i)|$ divides $(p^6 \pm p^3 - 1)$. Hence,  we  conclude that $q \mid (p^3 \pm 1)$ or $q \mid (p^6 \pm p^3 - 1)$ by Proposition \ref{reduction}, that is, $13 \mid (p^3 \pm 1)$ or $13 \mid (p^6 \pm p^3 - 1)$. Since $p \equiv 5 \pmod{q} \equiv 5 \pmod{13}$ and $\ell$ is either $5$ or $13$, we see that $\ell \not| (p^3 \pm 1)$ and $\ell \not| (p^6 \pm p^3 - 1)$, which is a contradiction. Hence there does not exist any point of order $5$ or $13$. This completes the proof.
\end{proof}

\smallskip

\begin{lemma}\label{s14}
The group $\mathbb{Z}/14\mathbb{Z}$ does not appear as $E(K)_{tors}$, for any $K$.
\end{lemma}
\begin{proof}
Suppose $E(K)_{tors} \cong \mathbb{Z}/14\mathbb{Z}$. In this case, we first observe that $E(K)[2] \cong \mathbb{Z}/2\mathbb{Z}$. Since  $E(K)_{tors}$ has a point of order $2$, we see that $c$ is a cube in $K$, say, $c = a^3$ for some $a \in K$. Since $E(K)_{tors}$ contains exactly one nontrivial point of order $2$,  by Lemma \ref{r2}, we get $\sqrt{-3} \notin K$. Let $P=(x,y)$ be a point of order $7$ in $E(K)_{tors}$. Then the corresponding division polynomial equation is
$$
(7t^2-4t+16)(t^6+564t^5-5808t^4-123136t^3-189696t^2-49152t+4096) = 0,
$$
where $t=\frac{x^3}{a^3} \in K$.

\smallskip

If $(7t^2-4t+16)=0$, then we get $t = \frac{2 \pm 6\omega}{7}$. This is a contradiction as $\omega \notin K$. Thus,
\begin{equation}\label{reduce}
t^6+564t^5-5808t^4-123136t^3-189696t^2-49152t+4096 = 0.
\end{equation}
Putting $t= s^3 = (\frac{x}{a})^3$ in  \eqref{reduce}, we get
$$
s^{18}+564s^{15}-5808s^{12}-123136s^9-189696s^6-49152s^3+4096=0.
$$

\smallskip

Using {\it magma}, we conclude that the polynomial in $s$ variable is an irreducible polynomial over $\mathbb{Q}$, which is a contradiction as $s \in K$ and $[K:\mathbb{Q}] = 6$. Therefore there does not exist a point of order $7$ in $E(K)_{tors}$. Hence the group $\mathbb{Z}/14\mathbb{Z}$ does not appear as a torsion subgroup in $E(K)$.
\end{proof}

\smallskip

\begin{lemma}\label{s18}
The group $\mathbb{Z}/18\mathbb{Z}$ does not appear as $E(K)_{tors}$, for any $K$.
\end{lemma}
\begin{proof}
Proof of this lemma is similar to the proof of Lemma \ref{K18} and we omit the proof here.
\end{proof}

\smallskip

\begin{lemma}\label{s36}
The group $\mathbb{Z}/3\mathbb{Z} \oplus \mathbb{Z}/6\mathbb{Z}$ does not appear as $E(K)_{tors}$, for any $K$.
\end{lemma}
\begin{proof}
At first we observe that $E(K)[2] \not\cong \{\mathcal{O}\}$ and hence $c$ is a cube in $K$.
Since $E(K)_{tors}$ contains eight nontrivial elements of order $3$, by Lemma \ref{r3} we conclude that $\omega \in K$. Since $\omega \in K$ and $c$ is a cube in $K$, by Lemma \ref{r2} we see that $E(K)_{tors}$ has three points of order $2$. This is a contradiction as there is only one nontrivial point of order $2$ in $\mathbb{Z}/3\mathbb{Z} \oplus \mathbb{Z}/6\mathbb{Z}$. 
\end{proof}

\smallskip

\begin{lemma}\label{s19}
The groups $\mathbb{Z}/19\mathbb{Z}, \mathbb{Z}/7\mathbb{Z}$ and  $\mathbb{Z}/2\mathbb{Z} \oplus \mathbb{Z}/14\mathbb{Z}$  appear as a torsion subgroup of $E(K)$ for some $K$.
\end{lemma}
\begin{proof}
{\bf Case 1:} ($\mathbb{Z}/19\mathbb{Z}$) 

Consider $K = \mathbb{Q}[e]/<e^6+e^4-e^3-2e^2+e+1>$ and take elliptic curve $E$ over $K$ in Kubert-Tate Normal form as 
$$E(2e^5-e^4+2e^3-4e^2+2, 2e^5-2e^4+4e^3-4e^2-2e+3).$$
Then by the equation \eqref{kub}, $E$ is a Mordell curve over $K$. In this case,    $E(K)_{tors} \simeq \mathbb{Z}/19\mathbb{Z}$ by \cite{ccrs}.

\smallskip

{\bf Case 2:} ($\mathbb{Z}/7\mathbb{Z}$)

Consider $K = \mathbb{Q}(2^{\frac{1}{3}}, \omega)$, where $\omega$ is a cube root of unity. Take elliptic curve $E$ over $K$ in Kubert-Tate Normal form as 
$E(\omega, -1)$. Then by the equation \eqref{kub}, we see that $E$ is a Mordell curve over $K$. Hence, in this case, $E(K)_{tors} \cong \mathbb{Z}/7\mathbb{Z}$ by using {\it magma}.

\smallskip

{\bf Case 3:} ($\mathbb{Z}/2\mathbb{Z} \oplus \mathbb{Z}/14\mathbb{Z}$)

Consider $K = \mathbb{Q}((6\omega + 30)^{\frac{1}{3}}))$, where $\omega$ is a cube root of unity. Take elliptic curve $E: y^2=x^3- \frac{\omega + 5}{36}$ over $K$. In this case, we have $E(K)_{tors} \cong \mathbb{Z}/2\mathbb{Z} \oplus \mathbb{Z}/14\mathbb{Z}$ by using {\it magma}.
\end{proof}

\smallskip

\noindent{\it Proof of Theorem \ref{cubicK}}. Combining Lemmas \ref{s4}, \ref{s5}, \ref{s14}, \ref{s18}, \ref{s36} and \ref{s19}, we have the desired result. \qed

\bigskip

\begin{acknowledgement*} We are thankful to Prof. Enrique Gonz{\'a}lez-Jim{\' e}nez for his helpful comments which has improved the presentation of this article. We would  like to sincerely thank Prof. R. Thangadurai for his continuous encouragement to finish this project. First author's research is supported by NBHM post-doctoral fellowship and Second author's research is supported by Institute fellowship of Harish-Chandra Research Institute.
\end{acknowledgement*}

\end{document}